\documentclass[a4paper,12pt]{article}
\usepackage{}
\usepackage{mathrsfs}
\usepackage{amssymb}
\usepackage{amsmath}
\usepackage{amsmath,amssymb,amsthm,graphics, graphicx,latexsym,amsfonts}
\usepackage{fancyhdr}
\usepackage{color}
\title{{Flips on homologous orientations of surface graphs with prescribed forbidden  facial circuits}
}
\author{{ Weijuan Zhang$^{1}$,
          Jianguo Qian$^{2}$\footnote{Corresponding author, email address: jgqian@xmu.edu.cn}}\\
\small  $^{1}$School of Mathematical Sciences, Xinjiang Normal University\\
\small  Urumqi, Xinjiang 830054, P.R.China\\
\small  $^{2}$School of Mathematical Sciences,  Xiamen University\\
\small  Xiamen, Fujian 361005, P.R.China}
\date{}

\usepackage{indentfirst}
\begin{document}
\maketitle

%\begin{CJK*}{GBK}{song}
\newtheorem{lem}{Lemma}[section]
\newtheorem{thm}[lem]{Theorem}
\newtheorem{prop}[lem]{Proposition}
\newtheorem{cor}[lem]{Corollary}
\newtheorem*{pf}{Proof}

\begin{abstract} Let $G$ be a graph embedded on an orientable surface. Given  a class ${\cal C}$ of facial circuits  of $G$ as a forbidden class, we give a sufficient-necessary condition for that an $\alpha$-orientation (orientation with prescribed out-degrees) of $G$  can be transformed into another by a sequence of flips on non-forbidden circuits and further give an explicit formula for the minimum number of such flips. We also consider the connection among all $\alpha$-orientations by defining a directed graph ${\bf D}({\cal C})$, namely the ${\cal C}$-forbidden flip graph. We show that if ${\cal C}\not=\emptyset$, then ${\bf D}({\cal C})$ has exactly $|O(G,{\cal C})|$ components, each of which is the cover graph of a distributive lattice, where $|O(G,{\cal C})|$ is the number of the $\alpha$-orientations that has no counterclockwise facial circuit other than that in ${\cal C}$. If ${\cal C}=\emptyset$, then every component of ${\bf D}({\cal C})$ is strongly connected. This generalizes the corresponding results of Felsner and Propp for the case that ${\cal C}$ consists of a single facial circuit.\\

\end{abstract}
\section {\large Introduction}

Graph orientation with certain degree-constraints on vertices is a natural focus in graph theory and has close  connection with many combinatorial structures in graphs, such as the spanning trees \cite{Felsner}, primal-dual orientations \cite{Disser},  bipolar orientations \cite{Fraysseix2}, transversal structures \cite{Fusy},  Schnyder woods \cite{Goncalves}, bipartite perfect matchings (or more generally, bipartite $f$-factors) \cite{Kenyon,Lam,Propp} and $c$-orientations of the dual of a plane graph \cite{Knauer,Propp}. In general, all these constraints could be  encoded in terms of  $\alpha$-orientations \cite{Felsner}, that is, the orientations with prescribed out-degrees.

To deal with the relation among orientations, circuit (or cycle) reversal has been shown as a useful method since it preserves the out-degree of each vertex and the connectivity of the orientations. Various types of cycle reversals were introduced subject to certain requirements. Circuit reversal for the orientations of plane graphs or, more generally, {\it surface graphs} (graphs embedded on an orientable surface) received particular attention \cite{Felsner,Goncalves,Knauer,Lam,Nakamoto}. For example, Nakamoto \cite{Nakamoto}  considered the 3-cycle reversal to deal with those orientations in plane triangulations where each vertex on the outer facial cycle has out-degree 1 while each of the other vertices has out-degree 3. In \cite{ZhangF}, Zhang et al. introduced the Z-transformation to study the connection among perfect matchings in hexagonal systems, which was later extended to general plane bipartite graphs \cite{Zhang2}.

A natural considering  of circuit  reversal for orientations of a surface graph is to reverse a directed facial circuit (the circuit that bounds a face). However, an orientation of a surface graph does not always have such a directed facial circuit, even if it has many directed circuits. In \cite{Felsner}, Felsner introduced a type of circuit reversals, namely  the {\it flip}, defined on the essential cycles in plane graphs which reverses a counterclockwise essential cycle to be clockwise. In a strongly connected $\alpha$-orientation  of a plane graph,  every essential cycle is exactly an inner facial circuit. In this sense, the notion of essential cycle is a very nice generalization of that of facial circuit. Further, Felsner proved that any $\alpha$-orientation of a plane graph can be transformed into a particular $\alpha$-orientation by a flip sequence, and the set of all $\alpha$-orientations of the graph carries a distributive lattice by flips.

The notion of flip has been also extended to  general surface graphs  \cite{Goncalves,Propp} and oriented matroids \cite{Knauer}. Moreover, Propp proved the following result (where the notion of homology will be given in the next section).
\begin{thm}\cite{Propp} \label{Propp} The set of all $\alpha$-orientations of a surface graph $G$ in the same homology class carries a distributive lattice by a sequence of flips taking on the facial circuits of $G$ with a prescribed forbidden  facial circuit.
\end{thm}
We note that the above theorem is an extension of the one given by Felsner for plane graphs since a plane graph can be regarded as a sphere graph where the outer facial circuit is treated as the forbidden circuit.

 In this paper, in stead of a  prescribed forbidden facial circuit, we consider the flips on surface graph with a class of prescribed forbidden facial circuits ${\cal C}$, and call such flips the ${\cal C}$-{\it forbidden flips}. In the third section, we give a sufficient and necessary condition for that an $\alpha$-orientation  can be transformed into another by a sequence of ${\cal C}$-forbidden flips and, further, give an explicit formula for the minimum number of the flips needed in such a sequence. In particular, if ${\cal C}$ is empty then an $\alpha$-orientation  can be transformed into another by flips if and only if they are homologous. In our study, the idea of potential function \cite{Felsner}, or depth function \cite{Nakamoto}, plays an important role.

In the last section, we focus on the relationship among all $\alpha$-orientations in a surface graph $G$. To this end, given a class  ${\cal C}$ of forbidden facial circuits, we define  a directed graph, namely the ${\cal C}$-forbidden flip graph, whose vertex set is the class of all $\alpha$-orientations of $G$ and two orientations $D$ and $D'$ are joined by an edge with direction  from $D'$ to $D$ provided $D$ can be transformed from $D'$ by a ${\cal C}$-forbidden flip. We apply the technique of U-coloring and L-coloring introduced by Felsner and Knauer in \cite{Felsner3}, and show that the ${\cal C}$-forbidden flip graph admits both a U- and L-coloring for any class  ${\cal C}$. This yields a generalization of Theorem \ref{Propp}. More specifically, if ${\cal C}$ is not empty, then the flip graph has exactly $|O(G,{\cal C})|$ components, each of which is the cover graph of a distributive lattice, where $|O(G,{\cal C})|$ is the number of the $\alpha$-orientations that has no counterclockwise facial circuit other than that in ${\cal C}$. And if ${\cal C}$ is empty, then every non-trivial component of the flip graph is strongly connected, and hence cyclic.
\section{Preliminaries}
In this section we introduce some elementary concepts involving graphs and graph embeddings on surfaces. For a  graph $G$ (directed or undirected), we denote by $V(G)$ and $E(G)$ the vertex set and edge set of $G$, respectively. A {\it closed walk} in a graph is a sequence $v_1v_2\dotsc v_{k}$ such that $v_1$ and $v_{k}$ are adjacent and, for every $i\in\{1,2,\dotsc,k-1\}$, $v_i$ and $v_{i+1}$ are adjacent. A closed walk  is called a {\it cycle} if it is edge disjoint and a {\it circuit} if it is vertex disjoint. The {\it directed circuit} and {\it directed cycle} are defined analogously. An {\it orientation} of a graph $G$ is an assignment of a direction to each edge of $G$. Given  a graph $G$ with $n$ vertices $v_1,v_2,\dotsc,v_n$ and an out-degree sequence $\alpha=(\alpha(v_1),\alpha (v_2),\dotsc, \alpha (v_n))$, an orientation $D$  of $G$ is called an $\alpha$-{\it orientation} \cite{Felsner} if  $d^+_D(v_i)=\alpha (v_i)$ for all $v_i\in V(G)$, where $d^+_D(v_i)$ is the out-degree of $v_i$ in $D$.  It is well known that if an $\alpha$-orientation of $G$ is strongly connected, then every $\alpha$-orientation of $G$ is strongly connected. In this case, we call $\alpha$ {\it strongly connected}. It is also known that each component obtained from an $\alpha$-orientation by deleting all those edges that have the same direction in every $\alpha$-orientation of $G$  (i.e., the {\it rigid edges} \cite{Felsner,Goncalves}) is strongly connected \cite{Knauer}. So in this paper, we always assume that $\alpha$ is  strongly connected and, hence, $G$ is 2-edge connected.

 A surface graph  is  a graph $G$ embedded on an orientable surface $\Sigma$ without boundary.  The connected components of $\Sigma\setminus G$, when viewed as subsets of $\Sigma$, are called the faces of $G$. A closed walk that bounds a face is called a {\it facial walk}.  A surface graph is called a {\it 2-cell embedding} \cite{Hutchinsonon} or {\it map} \cite{Goncalves} if all its faces are homeomorphic to open disks. A cycle $C$ on a surface $\Sigma$ is {\it separating} if $\Sigma\setminus C$ is disconnected and is  {\it contractible} \cite{Hutchinsonon} if $C$ can be continuously transformed into a single point. We note that if a cycle is contractible, then it is separating and, moreover, one component of  $\Sigma\setminus C$ is homeomorphic to an open disc. In this paper, except where otherwise stated, we always assume that all the surface graphs are 2-cell embedded and every edge is on the boundary of two distinct faces. We note that, in such an embedding, a facial walk is exactly a facial circuit and, therefore, every edge is shared by two distinct facial circuits.  Thus, our embedding is stronger than 2-cell embedding, but weaker than the one in which each facial circuit is contractible.

A directed facial circuit $F$ is called {\it counterclockwise} (resp., {\it clockwise}), or ccw (resp., cw) for short, if  when we trace along with the orientation of $F$, the face bounded by $F$ lies to the left (resp., right) side of $F$.  A {\it flip} taking on a  ccw facial circuit $F$ is the reorientation of $F$ that reverses $F$ from ccw to cw  \cite{Felsner}. The notion of flip was introduced  initially for the essential circuits.  As mentioned earlier, in a strongly connected $\alpha$-orientation of $G$, every essential circuit is  a facial circuit \cite{Knauer}.

 Let $G$ be a surface graph with edge set $E$ of $m$ edges and let $D_0$ be a fixed orientation of $G$. Consider the $m$-dimensional edge space $\mathbb{Z}^{|E|}$ of $G$. It is clear that any oriented subgraph $D$ of $G$ can be represented as an $m$-dimensional vector $\phi(D)\in\mathbb{Z}^{|E|}$ whose  coordinate indexed by an edge $e$ is defined by
\[\phi(D)_e=
\begin{cases}
 \ \ \ 1,& \text{if\ $e$\ has\ the\ same\ orientation\ in\ $D_0$\ and\ $D$},\\
 \ -1,& \text{if\ $e$\ has\ the\ opposite\ orientation\ in\ $D_0$\ and\ $D$},\\
 \ \ \  0,& \text{if\ $e$\ is\ not\ in\ $D$}.
\end{cases}
\]
Moreover, all oriented {\it even subgraphs} (the in-degree of each vertex equals its out-degree) of $G$ form an  $(m-n+1)$-dimensional subspace of $\mathbb{Z}^{|E|}$, called the (directed) {\it cycle space} of $G$, where $n$ is the number of vertices in $G$. 

For a simple graph $G$, a basis of the cycle space can be created by the elementary cycles associated with a given spanning tree of $G$. For a graph (not necessarily simple) embedded on a surface, it would be more convenient to create a basis of the cycle space by the facial circuits in ${\cal F}\setminus\{F_0\}$ and a so-called {\it basis for the homology}, where and herein after,  ${\cal F}$ denotes the set of all facial circuits  of $G$ and $F_0$ is an arbitrary facial circuit in ${\cal F}$. A common approach to create a basis for the homology  is to choose a spanning tree $T$ in $G$ and a spanning tree $T^*$ in the dual graph $G^*$ of $G$ that contains no edge dual to $T$. By Euler's formula, there are exactly $2g$ edges in $G$ that are neither in $T$ nor dual to the edges in $T^*$, where $g$ is the genus of the surface. Each of these $2g$ edges forms a unique circuit with $T$ and all these $2g$ circuits form a basis for the homology \cite{Goncalves},  denoted by ${\cal H}$. In a word, any oriented even subgraph $D$ of $G$ can be uniquely represented as a linear combination of the facial circuits in ${\cal F}\setminus\{F_0\}$ and circuits in ${\cal H}$, i.e.,
\begin{equation}\label{basis}
\phi(D)=\sum_{F\in{\cal F}\setminus\{F_0\}}\lambda_F \phi(F)+\sum_{H\in{\cal H}}\mu_H \phi(H).
\end{equation}
 In the following we always choose the orientation of each facial circuit in ${\cal F}$  to be ccw. Let $\mathbb{F}$ be the subspace of $\mathbb{Z}^{|E|}$ generated by ${\cal F}$. An oriented even subgraph $D$ is called {\it null-homologous} or {\it 0-homologous} \cite{Goncalves} if  $\phi(D)\in \mathbb{F}$, i.e.,
\begin{equation}\label{ccw}
\phi(D)=\sum_{F\in{\cal F}}\lambda_F \phi(F).
\end{equation}

 For two $\alpha$-orientations $D$ and $D'$, we denote by $D'-D$ the directed subgraph obtained from $D'$ by removing those edges that has the same direction as that in $D$. It is clear that $D'- D$ is oriented even graph.  Further, if $\phi(D'-D)\in\mathbb{F}$ then  $D$ and $D'$ are called {\it homologous} \cite{Goncalves}. Equivalently, $D$ and $D'$ are homologous if and only if their corresponding coefficients $\mu_H$ in (\ref{basis}) are the same for every $H\in{\cal H}$. It is clear that homology is an equivalence relation on $\alpha$-orientations.

\noindent{\bf Example 1}. Let $G$ be the Cartesian product graph  of two circuits of length 2 embedded on torus and $\alpha=(2,2,2,2)$. For convenience, we draw $G$ in the form as indicated in Figure 1 (left), where the left side and right side (resp., lower side and upper side)  are identified. By the method we introduced earlier, we can see that ${\cal H}=\{H_{\rm M},H_{\rm P}\}$  is a basis for the homology, where $H_{\rm M}$ and  $H_{\rm P}$ are the two directed circuits as illustrated in  Figure 1 (right).
\begin{center}
\scalebox{1.1}{\includegraphics{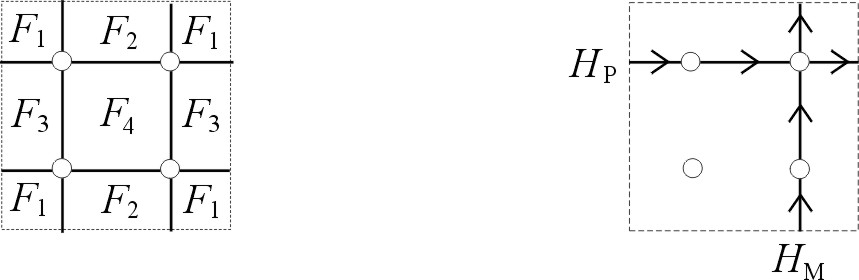}}\\
{\textbf{Figure\ 1}. The  edges of $G$ are represented in thick lines; the four faces are labeled by $F_1,F_2,F_3,F_4$.}
\end{center}
We note that, since every vertex has out-degree 2, the directions of the four edges on $F_1$ are determined uniquely by that  on $F_4$, unless  $F_4$ is a directed circuit (in this case $F_1$ must be directed and, therefore, has exactly two choices). This means that  $G$ has totally 18 different $(2,2,2,2)$-orientations $D_1,D_2,$ $\dotsc,D_{18}$, as illustrated in Figure 2. Since $D_{1}-D_{2}=H_{\rm M}$, we have $\phi(D_{1}-D_{2})=\phi(H_{\rm M})$, meaning that $D_{1}$ and $D_{2}$ are not homologous. Further, since $D_{13}-D_{14}=D_{13}$ and $\phi(D_{13})=\phi(F_1)+\phi(F_4)$, $D_{13}$ and $D_{14}$ are homologous.  In general, it can be seen that any two orientations $D_i$ and $D_j$ are homologous if and only if $i=j$ or $i,j\in\{13,14,\dotsc,18\}$. In other words, the $(2,2,2,2)$-orientations of $G$ form 13 homologous classes: $\{D_{1}\},\{D_2\},\dotsc,\{D_{12}\}$ and $\{D_{13},D_{14},\dotsc,D_{18}\}$.
\begin{center}
\scalebox{0.84}{\includegraphics{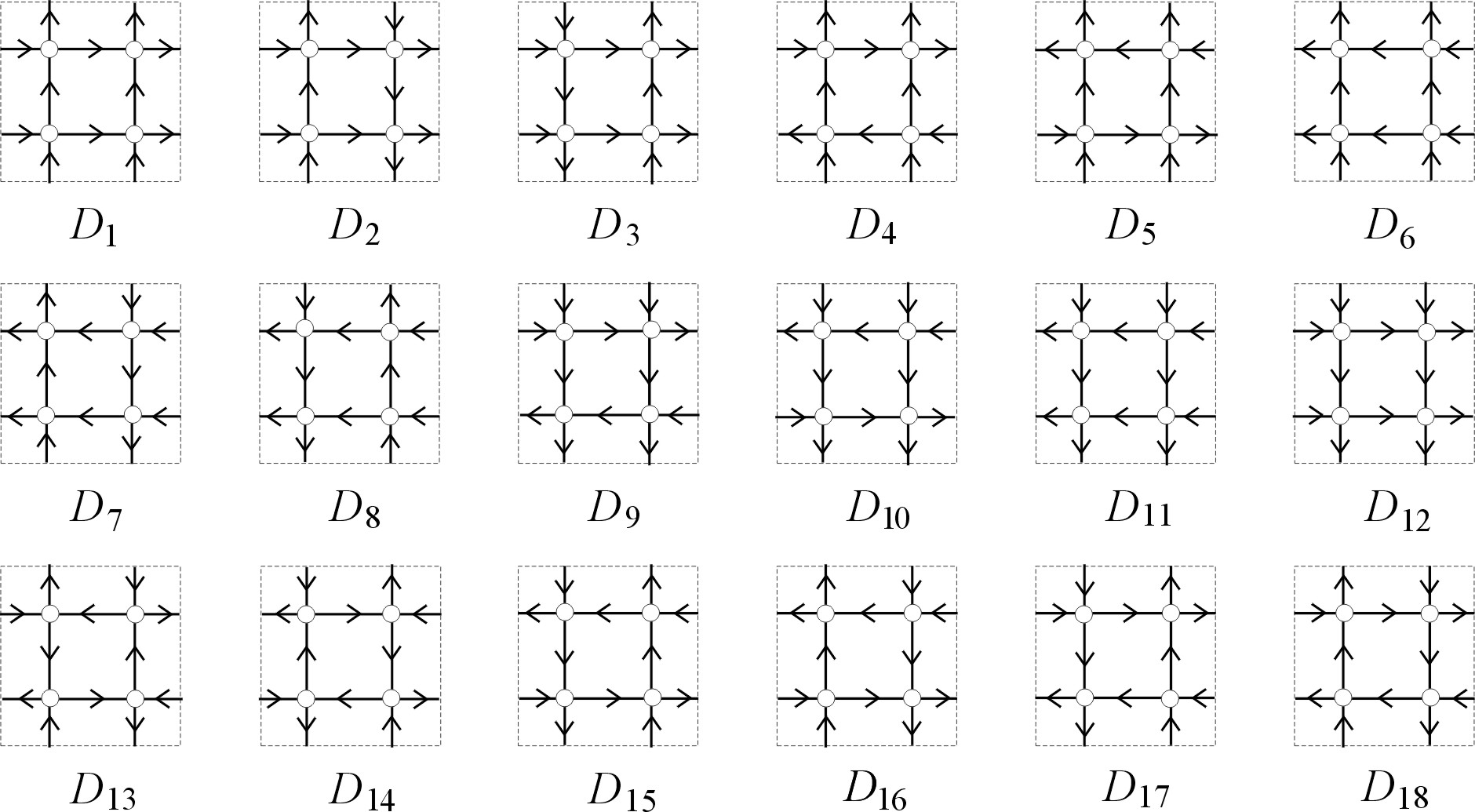}}\\
{\textbf{Figure\ 2}. }
\end{center}
\section{Transform an $\alpha$-orientation into another by flips on non-forbidden  facial circuits}
\begin{prop}\label{coordinate} Let $G$ be a surface graph  and, for each $F\in{\cal F}$, let $\lambda_F$ be an integer associated with $F$. Then for any integer $k$,
\begin{equation}
\sum_{F\in\mathcal{F}}\lambda_F \phi(F)=\sum_{F\in\mathcal{F}}(\lambda_F+k) \phi(F).
\end{equation}

Conversely, for each $F\in{\cal F}$, if  $\eta_F$ is an integer associated with $F$ such that
\begin{equation}\label{conv}
\sum_{F \in \mathcal{F}} \eta_F \phi(F) = \sum_{F \in\mathcal{F}} \lambda_F \phi(F),
\end{equation}
then there exists an integer $k$ such that $\eta_F = \lambda_F + k$ for every $F \in \mathcal{F}$.
\end{prop}
\begin{proof} Consider an edge $e$ of $G$. By our assumption, $e$ is shared by two distinct facial circuits, say, $F$ and $F'$. Recall that every facial circuit in ${\cal F}$, when viewed as a basis circuit, is ccw. Therefore, $e$ has opposite orientations in $F$ and $F'$. Thus, $k\phi(F)_e+k\phi(F')_e=0$ and $\phi(K)_e=0$ for any $K\in{\cal F}$ other than $F$ and $F'$. Thus, (3) follows directly.

Conversely, let $e$ be an arbitrary edge and let $F$ and $F'$ be the two facial circuits that share $e$. Then by (\ref{conv}),
\begin{equation}\label{mulambda}
\mu_F \phi(F)_e+\eta_{F'} \phi(F')_e=\lambda_F \phi(F)_e+\lambda_{F'} \phi(F')_e.
\end{equation}
Further, we assume that  $\eta_F = \lambda_F + k$ and  $\eta_{F'} = \lambda_{F'} + k'$. Then,
 $$\eta_F \phi(F)_e+\eta_{F'} \phi(F')_e=(\lambda_F+k)\phi(F)_e+(\lambda_{F'}+k')\phi(F')_e$$
 $$=(\lambda_F \phi(F)_e+\lambda_{F'} \phi(F')_e)+(k\phi(F)_e+k'\phi(F')_e.$$
So by (\ref{mulambda}), we have $k\phi(F)_e+k'\phi(F')_e=0$. Recall that both $F$ and $F'$ are ccw. Therefore, $\phi(F)_e=-\phi(F')_e$ and, hence, $k=k'$. Further, notice that the dual graph of $G$ is connected, meaning that $\eta_F = \lambda_F + k$ for all $F \in \mathcal{F}$.
 \end{proof}

Let  $D$ and  $D'$ be two homologous $\alpha$-orientations  of a surface graph $G$. Following the idea of potential function \cite{Felsner} (or depth function \cite{Nakamoto}), we introduce an intuitive representation of $\phi(D'-D)$ as follows: define
$$z_{D'-D}:{\cal F}\rightarrow \mathbf{Z}$$
to be the function which assigns an integer to each facial circuit  $F$ of $G$ according to the following rule, see Figure 1:\\
1.\ $z_{D'-D}(F_0)=0$;\\
2.\ if $F$ and $F'$ share a common edge $e$ then
\begin{equation}\label{pt}
z_{D'-D}(F)=
\begin{cases}
 \ z_{D'-D}(F')+1,& \text{if\ $e\in D'-D$\ and\ $F$\ lies\ to\ the\ left\ of\ $e$},\\
 \ z_{D'-D}(F')-1,& \text{if\ $e\in D'-D$\ and\ $F'$\ lies\ to\ the\ left\ of\ $e$},\\
 \ z_{D'-D}(F'),& \text{if\ $e\notin D'-D$},\\
\end{cases}
\end{equation}
where, by `$F$ lies to the left of $e$'  we mean that $F$ lies to the left side of $e$ when we trace along with the direction of $e$.

\begin{center}
\scalebox{0.45}{\includegraphics{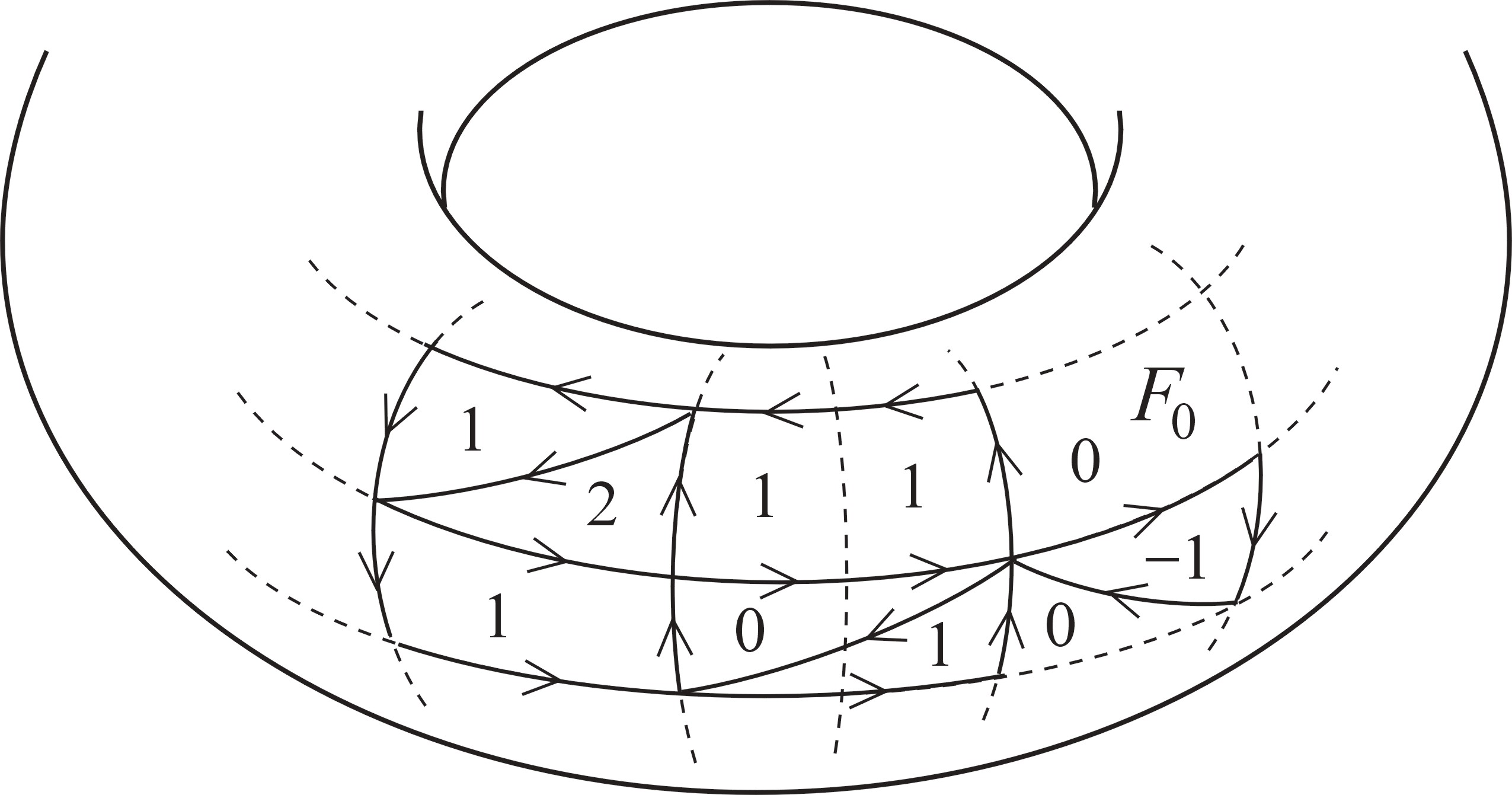}}\\
{\textbf{Figure\ 3}. The edges in $D'-D$ are represented in thick lines and  edges not in $D'-D$ are represented in dotted lines.}
\end{center}

\begin{prop}\label{potential} For any two homologous $\alpha$-orientations $D$ and $D'$,  $z_{D'-D}(F)$ is well defined and
\begin{equation*}
\phi(D'-D)=\sum_{F\in\mathcal{F}}z_{D'-D}(F) \phi(F).
\end{equation*}
\end{prop}
\begin{proof}Since $D$ and  $D'$ are homologous, we have $\phi(D'-D)=\sum_{F\in{\cal F}}\lambda_F \phi(F)$. Let $e$ be an arbitrary edge of $G$, and let $F$ and $F'$ be the two facial circuits that share $e$. Then $\phi(D'-D)_e=\lambda_F \phi(F)_e+\lambda_{F'} \phi(F')_e$.

We first assume that the direction of $e$ in $D'-D$ coincides with that in our initially fixed orientation $D_0$, i.e.,  $\phi(D'-D)_e=\phi(D_0)_e=1$. In this case, we have $\phi(F)_e=1$ and  $\phi(F')_e=-1$ if $F$ lies to the left of $e$, and $\phi(F)_e=-1$ and  $\phi(F')_e=1$ if $F$ lies to the right of $e$. Therefore,
\begin{equation}\label{phi}
\lambda_F=
\begin{cases}
 \ \lambda_{F'}+1,& \text{if\ $e\in D'-D$\ and\ $F$\ lies\ to\ the\ left\ of\ $e$},\\
 \ \lambda_{F'}-1,& \text{if\ $e\in D'-D$\ and\ $F'$\ lies\ to\ the\ left\ of\ $e$},\\
 \ \lambda_{F'},& \text{if\ $e\notin D'-D$}.\\
\end{cases}
\end{equation}

We now assume that $\phi(D'-D)_e=-\phi(D_0)_e=-1$. In this case we have $\phi(F)_e=-1$ and  $\phi(F')_e=1$ if $F$ lies to the left of $e$, and $\phi(F)_e=1$ and  $\phi(F')_e=-1$ if $F$ lies to the right of $e$. This means that (\ref{phi}) still holds.

Combining (\ref{phi}) with (\ref{pt}), we have
 \begin{equation}\label{well}
\lambda_F-\lambda_{F'}=z_{D'-D}(F)-z_{D'-D}(F')
\end{equation}
and, hence, $\lambda_F-z_{D'-D}(F)=\lambda_{F'}-z_{D'-D}(F')$. Since the dual graph of $G$ is connected, we have
\begin{equation}\label{cp}
\lambda_F-z_{D'-D}(F)=\lambda_{K}-z_{D'-D}(K)
\end{equation}
 for any $K\in {\cal F}$. In particular, $\lambda_F-z_{D'-D}(F)=\lambda_{F_0}-z_{D'-D}(F_0)=\lambda_{F_0}$, i.e., $z_{D'-D}(F)=\lambda_F-\lambda_{F_0}$. Therefore, $z_{D'-D}(F)$ is well-defined. Further, the proposition follows directly from (\ref{cp}) and Proposition \ref{coordinate}.
\end{proof}
\begin{lem}\label{base}  \cite{Goncalves} Let $D$ and $D'$ be two $\alpha$-orientations of a  surface graph and let  $F_0$ be an arbitrary facial circuit. If
\begin{equation*}
\phi(D'-D)=\sum_{F\in\mathcal{F}\setminus\{F_0\}}\lambda_F \phi(F)
\end{equation*}
and $\lambda_F\geq 0$ for all $F\in\mathcal{F}\setminus\{F_0\}$, then $D$ can be transformed from $D'$ by a sequence of $\{F_0\}$-forbidden flips.
\end{lem}

Let
$$z_{\min}=\min\{z_{D'-D}(F):F\in{\cal F}(G)\}$$
and let $F_{\min}$ be a facial circuit for which $z_{D'-D}(F_{\min})=z_{\min}$.

\begin{thm}\label{main} Given  a  forbidden facial circuit class ${\cal C}$, an $\alpha$-orientation $D$ of a surface graph can be transformed from $D'$ by a sequence of ${\cal C}$-forbidden flips if and only if $D$ and $D'$ are homologous and $z_{D'-D}(F)=z_{\min}$ for every $F\in{\cal C}$. Moreover, the minimum number of flips needed to transform $D'$ into $D$ equals
\begin{equation}
d(D',D)=\sum\limits_{F\in {\cal F}}(z_{D'-D}(F)-z_{\min}).
\end{equation}
\end{thm}
\begin{proof} Assume firstly that  $D$ can be transformed from $D'$ by a sequence ${\cal S}$ of ${\cal C}$-forbidden flips.

For a facial circuit $F\in{\cal F}$, consider the number $t(F)$ of the flips in ${\cal S}$  taking on $F$. Let $e$ be an edge on $F$ and let $F'$ be the facial circuit which shares $e$ with $F$. If $e\in D'-D$, then the orientation of $e$ changes after ${\cal S}$ being taken. Moreover, a flip to change the orientation of $e$ must take on the facial circuit which lies to the left side of $e$. This implies that $t(F)=t(F')+1$ if  $F$  lies to the left  side of $e$ or $t(F)=t(F')-1$  if $F'$ lies to the left  side of $e$. If $e\notin D'-D$, then the orientation of $e$ does not change after ${\cal S}$ being taken. This means that, if a flip in ${\cal S}$ takes on $F$ then there must be another flip that takes on $F'$ to keep the orientation of $e$ invariant and, hence,  $t(F)=t(F')$. As a result, we have the following relation:
\begin{equation}\label{t}
t(F)=
\begin{cases}
 \ t(F')+1,& \text{if\ $e\in D'-D$\ and\ $F$\  lies\ to\ the\ left\  side\ of\ $e$},\\
 \ t(F')-1,& \text{if\ $e\in D'-D$\ and\ $F'$\ lies\ to\ the\ left\  side\ of\ $e$},\\
 \ t(F'),& \text{if\ $e\notin D'-D$}.
\end{cases}
\end{equation}

Combining (\ref{t}) with (\ref{pt}), we have $t(F)-t(F')=z_{D'-D}(F)-z_{D'-D}(F')$, i.e., $t(F)-z_{D'-D}(F)=t(F')-z_{D'-D}(F')$. Since the dual graph of $G$ is connected, we have
\begin{equation}\label{tz}
t(F)-z_{D'-D}(F)=t(K)-z_{D'-D}(K)
\end{equation}
 for any $K\in {\cal F}$.  So by Proposition \ref{coordinate} and Proposition \ref{potential}, 
$$\phi(D'-D)=\sum_{F\in\mathcal{F}}t(F) \phi(F).$$
Hence, $D$ and $D'$ are homologous.

Further, from (\ref{tz}) we can also see that $t(K)$ attains the minimum value if and only if $z_{D'-D}(K)$ does, i.e., $z_{D'-D}(K)=z_{\min}$ and $K=F_{\min}$. On the other hand, notice that $t(F)\geq 0$ for every $F\in{\cal F}$ and the minimum value 0 of $t(F)$ is attained by every $F\in{\cal C}$ since the facial circuits in ${\cal C}$ do not involve any flip in ${\cal S}$. Hence, $t(F_{\min})=t(F)=0$ for every $F\in{\cal C}$. Therefore,
 \begin{equation}
 z_{D'-D}(F)=z_{D'-D}(F_{\min})=z_{\min}
  \end{equation}
for every $F\in{\cal C}$.

 Conversely, assume that $D$ and $D'$ are homologous and  $z_{D'-D}(F)=z_{\min}$  for any $F\in{\cal C}$. Then by Proposition  \ref{potential} and Proposition \ref{coordinate}, we have
\begin{eqnarray}
\phi(D'-D)
&=&\sum_{F\in\mathcal{F}}z_{D'-D}(F) \phi(F)\nonumber\\
&=&\sum_{F\in{\mathcal F}}(z_{D'-D}(F)-z_{D'-D}(F_{\min}))\phi(F)\nonumber\\
&=&\sum_{F\in{\mathcal F}\setminus\{F_{\min}\}}(z_{D'-D}(F)-z_{D'-D}(F_{\min}))\phi(F)\nonumber.
\end{eqnarray}

Since $z_{D'-D}(F) -z_{D'-D}(F_{\min})=z_{D'-D}(F) -z_{\min}\geq 0$ for any $F\in{\mathcal F}\setminus\{F_{\min}\}$, so by Lemma \ref{base},  $D$ can be transformed from $D'$ by a sequence of $\{F_{\min}\}$-forbidden flips.

Let ${\cal S}$ be an arbitrary sequence of $\{F_{\min}\}$-forbidden  flips that transform $D'$ into $D$.  For any facial circuit $F\in{\cal F}$, consider the number $t(F)$ of flips in ${\cal S}$  taking on $F$. By the definition of  ${\cal S}$, we have $t(F_{\min})=0$.

Similar to the proof of the necessity,  (\ref{tz}) is also satisfied.  Set $K=F_{\min}$ in  (\ref{tz}). 
Since $t(F_{\min})=0$ and $z_{D'-D}(F)=z_{\min}$  for every $F\in{\cal C}$, we have $t(F)=t(F_{\min})=0$ for every $F\in{\cal C}$. This implies that ${\cal S}$ contains no flip taking on a facial circuit in ${\cal C}$. That is,  ${\cal S}$ is a ${\cal C}$-forbidden sequence.

Finally, by the arbitrariness of the sequence ${\cal S}$, we may assume that ${\cal S}$ is shortest. Thus, by (\ref{tz}) and the fact that $t(F_{\min})=0$, we have
\begin{eqnarray}
d(D',D)
&=&\sum_{F\in {\cal F}}t(F)=\sum_{F\in {\cal F}}(t(F)-t(F_{\min}))\nonumber\\
&=&\sum_{F\in {\cal F}}(z_{D'-D}(F)-z_{\min}).\nonumber
\end{eqnarray}

This completes our proof. \end{proof}

\begin{cor}\label{transform} Let $D$ and $D'$ be two $\alpha$-orientations of a surface graph. Then $D$ can be transformed from $D'$ by a sequence of flips if and only if $D$ and $D'$ are homologous.
\end{cor}
\begin{proof} Recall that $z_{D'-D}(F_{\min})=z_{\min}$. So by Theorem \ref{main}, if $D$ and $D'$ are homologous then $D$ can be transformed from $D'$ by a sequence of $\{F_{\min}\}$-forbidden flips. The converse is also true by Theorem \ref{main}.
\end{proof}

We now give an extension of Lemma \ref{base} as follows.

\begin{cor} Let $D$ and $D'$ be two $\alpha$-orientations of a  surface graph and  ${\cal C}$ be a class of facial circuits. If
\begin{equation}\label{ext}
\phi(D'-D)=\sum_{F\in\mathcal{F}\setminus{\cal C}}\lambda_F \phi(F),
\end{equation}
 then $D$ can be transformed from $D'$ by a sequence of ${\cal C}$-forbidden flips if and only if $\lambda_F\geq 0$ for all $F\in\mathcal{F}\setminus{\cal C}$.
\end{cor}
\begin{proof}We write (\ref{ext}) as 
\begin{equation}\label{rew}
\phi(D'-D)=\sum_{F\in\mathcal{F}}\lambda_F \phi(F),
\end{equation}
 where $\lambda_F=0$ if $F\in{\cal C}$. By Proposition \ref{potential}, we have
$$\phi(D'-D)=\sum_{F\in\mathcal{F}}z_{D'-D}(F) \phi(F).$$
So by Proposition \ref{coordinate},  there exists an integer $k$
   such that $z_{D'-D}(F) = \lambda_F + k$ for all $F \in \mathcal{F}$. Further, the condition $\lambda_F\geq 0$ for all $F\in\mathcal{F}\setminus{\cal C}$ implies that  the minimum value of $\lambda_F$ in (\ref{rew}) is 0 and attained by every $F\in{\cal C}$. Hence, the minimum value $z_{\min}$ of $z_{D'-D}(F)$ is attained by  every $F\in{\cal C}$. The corollary follows directly from Theorem \ref{main}.
   \end{proof}

\section{Flip graph of $\alpha$-orientations}
For  a class ${\cal C}$ of forbidden facial circuits of a surface graph $G$, we define the {\it ${\cal C}$-forbidden flip graph} of $G$ as a directed graph, denoted by  ${\bf D}({\cal C})$, whose vertex set is the class of all $\alpha$-orientations of $G$, and two orientations $D$ and $D'$ are joined by a directed edge from $D'$ to $D$ provided $D$ can be transformed from $D'$ by a ${\cal C}$-forbidden flip, or equivalently, $D'-D$ is a ccw facial circuit and  $D'-D\notin {\cal C}$. Since the homology is an equivalence relation,  the $\alpha$-orientations of $G$ can be partitioned into several equivalence classes, say $\Omega_1,\Omega_2,\dotsc,\Omega_q$. For  $k\in\{1,2,\dotsc,q\}$, let ${\bf D}_k({\cal C})$ be the subgraph of ${\bf D}({\cal C})$ induced by $\Omega_k$. If ${\cal C}$ consists of exactly one facial circuit, then by Theorem \ref{Propp}, ${\bf D}_k({\cal C})$ is the cover graph of a distributive lattice.

In this section we consider the general case, in which the forbidden class ${\cal C}$ consists of any number of facial circuits (in particular,  ${\cal C}$ may be empty). To this end, we apply the technique of U-coloring and L-coloring, which was introduced by Felsner and Knauer to characterize the cover graph of distributive lattices.

For a directed graph $D=(V,E)$ and two vertices $u,v\in V$, we use $(u,v)$ to denote the directed edge joining $u$ and $v$ with direction from $u$ to $v$. An edge coloring $c$ of $D$ is called a {\it U-coloring}  if for every $u, v,w\in V$ with $u\not=w$ and $(v, u), (v,w)\in E$, the following two rules hold \cite{Felsner3}:\\
(U$_1$). $c(v, u)\not= c(v,w)$;\\
(U$_2$). There is a vertex $z\in V$ and edges $(u, z), (w, z)$  such that $c(v, u) = c(w, z)$ and $c(v,w) = c(u, z)$.

An L-coloring is defined as the dual, in the sense of edge direction reversal, to a U-coloring, that is, every directed edge $(i,j)$ in the definition of U-coloring is replaced by $(j,i)$.
\begin{prop}\label{UL}\cite{Felsner3} If a  finite connected acyclic digraph $D$ admits a U- and an L-coloring then $D$ is the cover graph of a distributive lattice, and has a unique sink and a unique source.
\end{prop}

We now turn to our flip graph. Let ${\cal F}=\{F_1,F_2,\dotsc,F_{|{\cal F}|}\}$ be the class of all facial circuits of $G$. An edge of ${\bf D}(\emptyset)$ is called of type $F_i$ if it corresponds to a flip taking on $F_i$. Let ${\bf E}_i$ be the set of the edges of type $F_i$ and let ${\bf c}:E({\bf D}(\emptyset))\rightarrow \{1,2,\dotsc,|{\cal F}|\}$ be the edge coloring of ${\bf D}(\emptyset)$ such that ${\bf c}({\bf e})=i$ for each ${\bf e}\in {\bf E}_i, i=1,2,\dotsc, |{\cal F}|$.
\begin{prop}\label{U-coloring} ${\bf c}$ is both a U-coloring and an L-coloring.
\end{prop}
\begin{proof} Let $D_1,D_2,D_3\in V({\bf D}(\emptyset))$ with $D_2\not=D_3$ and $(D_1,D_2)\in {\bf E}_i, (D_1,D_3)\in {\bf E}_j$. Therefore, $D_2$ (resp., $D_3$) can be transformed from $D_1$ by a flip taking on the ccw facial circuit $F_i$ (resp., $F_j$). Since  $D_2\not=D_3$, we have $i\not=j$ and, hence, the rule U$_1$ holds. Further, since both $F_i$ and $F_j$ are ccw, they cannot share the same edge of $G$, that is, $F_i$ and $F_j$ are independent. Let $D_4$ be the orientation of $G$ transformed from $D_2$ by taking a flip on $F_j$. Then $D_4$ can also be transformed from $D_3$ by taking a flip on $F_i$. This means that $(D_3,D_4)\in {\bf E}_i,(D_2,D_4)\in {\bf E}_j$ and, therefore, the rule U$_2$ holds. The discussion for L-coloring is analogous.
\end{proof}

  For a set {\bf E} of edges in ${\bf D}(\emptyset)$, we denote by ${\bf D}(\emptyset)-{\bf E}$ the directed graph obtained from  ${\bf D}(\emptyset)$ by deleting all the edges in {\bf E}. By the definition of the edge coloring ${\bf c}$, the following proposition is obvious.
\begin{prop}\label{heredity1} For any $I\subseteq\{1,2,\dotsc,|{\cal F}|\}$,
 ${\bf c}$ restricted on the subgraph ${\bf D}(\emptyset)-\bigcup_{i\in I}{\bf E}_{i}$  is still a U-coloring and an L-coloring.
\end{prop}

\begin{prop}\label{contraction}  For any facial circuit class ${\cal C}=\{F_{i_1},F_{i_2},\dotsc,F_{i_s}\}$, we have
\begin{equation}\label{8}
{\bf D}( {\cal C})={\bf D}(\emptyset)- ({\bf E}_{i_1}\cup{\bf E}_{i_2}\cup\cdots\cup{\bf E}_{i_s}).
\end{equation}
\end{prop}
\begin{proof} Notice that ${\bf D}( {\cal C})$ and ${\bf D}(\emptyset)- ({\bf E}_{i_1}\cup{\bf E}_{i_2}\cup\cdots\cup{\bf E}_{i_s})$  have the same vertex set. If an edge ${\bf e}$ is in some ${\bf E}_{i_j}$ then ${\bf e}$ is of type $F_{i_j}$ and therefore, ${\bf e}$ is not in  ${\bf D}( {\cal C})$ since $F_{i_j}$ is a forbidden facial circuit in ${\cal C}$. The converse is also true and thus (\ref{8}) follows.
\end{proof}

In contrast to the U- and L-colorings for the cover graph of a lattice, we will see that  ${\bf D}({\cal C})$ may be cyclic if ${\cal C}=\emptyset$, which gives a non-trivial example of cyclic directed graph that admits a U- and an L-coloring. As an example, let us consider the surface graph $G$ in Example 1. Recall that the equivalence classes, i.e., the homologous classes, of the $(2,2,2,2)$-orientations of $G$ are
$$\Omega_i=\{D_i\}\ \ {\rm for}\ i\in \{1,2,\dotsc,12\}\ \ {\rm and}\ \ \Omega_{13}=\{D_{13},D_{14},\dotsc,D_{18}\}.$$
This means that the flip graph ${\bf D}(\emptyset)$ of $G$ consists of 12 isolated vertices $D_1,D_2,\dotsc,D_{12}$ and the subgraph ${\bf D}_{13}(\emptyset)$ induced by $\Omega_{13}$, where  ${\bf D}_{13}(\emptyset)$ is illustrated as in Figure 4 and is cyclic. Further, by Proposition \ref{contraction}, we have ${\bf D}_{13}(\{F_4\})={\bf D}_{13}(\emptyset)- {\bf E}_{4}$ and ${\bf D}_{13}(\{F_1,F_4\})={\bf D}_{13}(\emptyset)- ({\bf E}_{1}\cup{\bf E}_{4})$. Moreover, ${\bf D}_{13}(\{F_4\})$ has one component and ${\bf D}_{13}(\{F_1,F_4\})$ has three components, each of which is the cover graph of a distributive lattice, see Figure 4.

\begin{center}
\scalebox{0.56}{\includegraphics{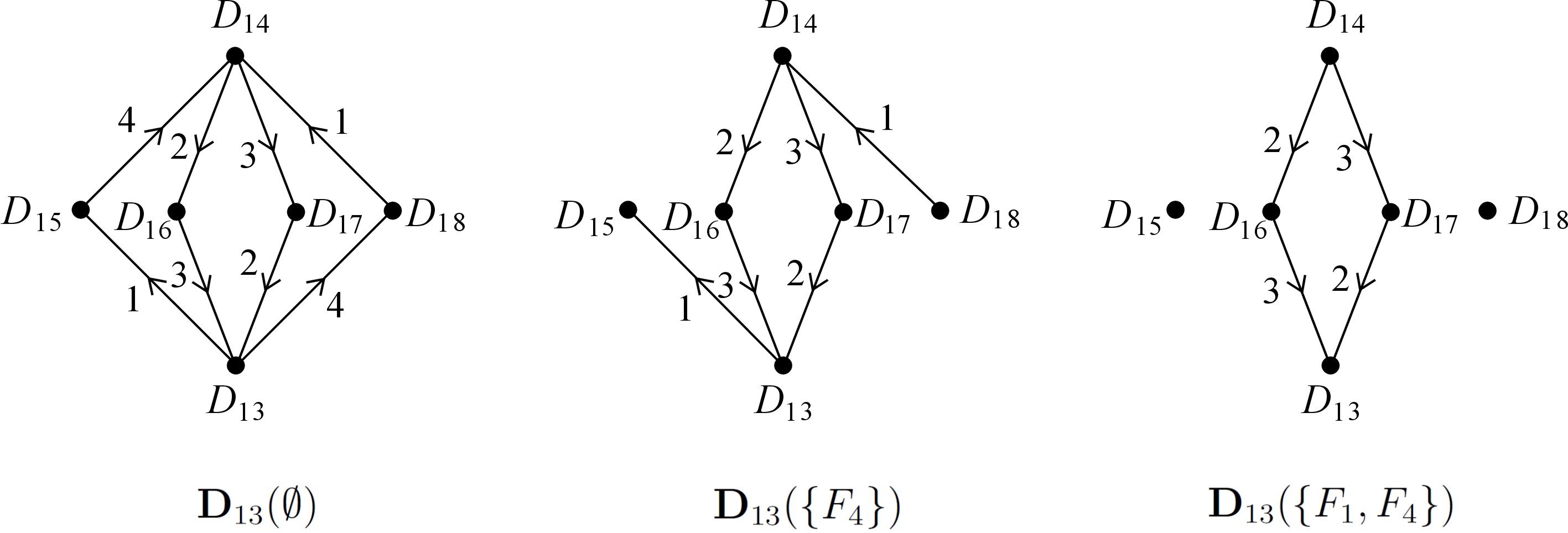}}\\
{\textbf{Figure\ 4}. For $i\in\{1,2,3,4\}$, the edges of type $F_i$ are represented by $i$.}
\end{center}

Recall that a directed graph $D$ is {\it strongly connected} if for any two vertices $u$ and $v$, there is a directed path from $u$ to $v$ and  a directed path from $v$ to $u$. The maximum integer $s$ for which $D\setminus H$ is strongly connected for any set $H$ of at most $s-1$ edges is called the {\it strong edge-connectivity} of $D$ and is denoted by $\kappa'(D)$.

For $k\in\{1,2,\ldots,q\}$, let $O_k(G,{\cal C})$ be the class of all the $\alpha$-orientations in $\Omega_k$ that has no ccw facial circuit other than that in ${\cal C}$ and let $O(G,{\cal C})=\bigcup_{k=1}^qO_k(G,{\cal C})$. Similarly, let $O^*_k(G,{\cal C})$ be the class of all the $\alpha$-orientations in $\Omega_k$ that has no cw facial circuit other than that in ${\cal C}$. We can see that, for any $D\in O_k(G,{\cal C})$, the orientation obtained from $D$ by reversing the directions of all edges of $D$ is in $O^*_k(G,{\cal C})$, and vise versa. This means that $|O_k(G,{\cal C})|=|O^*_k(G,{\cal C})|$. For instance, in the example above we have $O_{13}(G,\{F_1,F_4\})=\{D_{13},D_{15},D_{18}\}$ and $O^*_{13}(G,\{F_1,F_4\})=\{D_{14},D_{15},D_{18}\}$, see Figure 2. Further, we can see that $D_{13},D_{15},D_{18}$ and $D_{14},D_{15},D_{18}$ are exactly the three sinks and three sources of  ${\bf D}_{13}(\{F_1,F_4\})$, respectively, see Figure 4.
\begin{thm}\label{connectivty} Let $k\in\{1,2,\dotsc,q\}$.  \\
1). ${\bf D}_{k}(\emptyset)$ is strongly connected and $\kappa'({\bf D}_{k}(\emptyset))=1$;\\
2). For any ${\cal C}\subseteq{\cal F}$ with ${\cal C}\not=\emptyset$, ${\bf D}_{k}({\cal C})$ has exactly $|O_k(G,{\cal C})|=|O^*_k(G,{\cal C})|$ components, each of which is the cover graph of a distributive lattice with the unique sink in $O_k(G,{\cal C})$ and a unique source in $O^*_k(G,{\cal C})$.
\end{thm}
\begin{proof} 1). By Corollary \ref{transform}, ${\bf D}_{k}(\emptyset)$ is strongly connected and therefore, $\kappa'({\bf D}_{k}(\emptyset))\geq 1$. Let $F$ be an arbitrary facial circuit $F$ and let ${\cal L}(F)$ be the lattice  on the homology class $\Omega_k$ with forbidden circuit $F$, where two orientations $D$ and $D'$ have the relation $D'\leq D$ in ${\cal L}(F)$ if  $D$ can be transformed from $D'$ by a sequence of ${\cal C}$-forbidden flips. Then by Theorem \ref{Propp}, ${\bf D}_{k}(F)$ is the cover graph of ${\cal L}(F)$. Recall that ${\cal L}(F)$  has a unique maximal element. Moreover, this maximal element clearly corresponds to an orientation $D_{\max}$ of $G$ that contains no ccw facial circuit other than $F$ \cite{Goncalves}. On the other hand, since $\kappa'({\bf D}_{k}(\emptyset))\geq 1$,  $D_{\max}$ has at least one ccw facial circuit. Thus, $F$ itself must be ccw in $D_{\max}$. This means that $F$ is the unique ccw facial circuit in $D_{\max}$ and therefore, $\kappa'({\bf D}_{k}(\emptyset))=1$.

2). We first prove that ${\bf D}_{k}({\cal C})$ is acyclic. Suppose to the contrary that ${\bf D}_{k}({\cal C})$ has a directed cycle ${\bf C}=D_1D_2\cdots D_tD_1$. Without loss of generality, assume that the facial circuits sequence of the flips corresponding to ${\bf C}$ is ${\cal S}: F_1,F_2,\dotsc,F_t$. Since ${\bf C}$ is a directed cycle,  the direction of every edge in $D_1$ is not changed after taking the above flips. This means that, for every edge $e$  in $D_1$, the two facial circuits that share $e$ appear the same times in the sequence ${\cal S}$. Therefore,  every facial circuit in $D_1$ must appear the same times in ${\cal S}$ since the dual graph of $G$ is connected. This is a contradiction because the facial circuits in ${\cal C}$ are forbidden to any flip.

Finally, by Proposition \ref{U-coloring}, Proposition \ref{heredity1} and Proposition \ref{contraction}, ${\bf D}_{k}({\cal C})$ admits a U- and an L-coloring. So by Proposition \ref{UL}, each component of ${\bf D}_{k}({\cal C})$ is the cover graph of a distributive lattice (in particular, a single element) with a unique sink and a unique source. Moreover, if an orientation $D$ of $G$ is a sink of ${\bf D}_{k}({\cal C})$, then $D$ has no ccw facial circuit other than that in ${\cal C}$, meaning that $D\in O_k(G,{\cal C})$. The discussion for $D$ being the source is analogous. This completes our proof.
\end{proof}

\section*{Acknowledgments}
This work is supported by National Natural Science Foundation of China  under Grant No.\,11971406.

%\end{CJK*}
\end{document}